\documentclass[12pt]{article}
\usepackage{mathrsfs}
\usepackage{amsmath}
\usepackage{float}
\usepackage{cite}
\usepackage{esint}
\usepackage[all]{xy}
\usepackage{amssymb}
\usepackage{amsfonts}
\usepackage{amsthm}
\usepackage{color}
\usepackage{graphicx}
\usepackage{hyperref}
\usepackage{bm}
\usepackage{indentfirst}
\usepackage{geometry}
\theoremstyle{plain}
\newtheorem{thm}{Theorem}[section]

\newtheorem{defn}[thm]{Definition}
\newtheorem{prop}[thm]{Proposition}

\newtheorem{lem}[thm]{Lemma}

\newtheorem{them}[thm]{Theorem}
\theoremstyle{remark}
\newtheorem{ex}[thm]{Example}
\newtheorem{rmk}[thm]{Remark}
\theoremstyle{plain}
\newtheorem*{theoremoneprime}{Theorem 1.1$'$}

\newcommand{\Z}{\mathbb{Z}}

\newcommand{\R}{\mathbb{R}}


\numberwithin{equation}{section}

\title{The Smoothability of Embedded Surfaces in 4-Manifolds Under Stabilizations}
\author{Zuyi Zhang}
\date{ }

\begin{document}

\maketitle
\begin{abstract}
    We study the smooth stabilization number of a locally flat embedded surface in a smooth $4$-manifold, defined as the minimum number of connected sums with $S^2\times S^2$ needed to make the surface topologically isotopic to a smooth embedding. We prove that a nonsmoothable locally flat surface in a closed, simply connected smooth $4$-manifold becomes smoothable after a single stabilization if its homology class is primitive and noncharacteristic and its complement has abelian fundamental group. This yields infinitely many examples of nonsmoothable surfaces with smooth stabilization number one. We also exhibit examples for which a single stabilization produces infinitely many pairwise smoothly nonisotopic representatives in the same topological isotopy class. In contrast, we construct a locally flat genus-eight surface representing $6h$ in $\mathbb{CP}^2$ whose smooth stabilization number is exactly two. More generally, we show that smooth stabilization numbers of locally flat embedded surfaces in $\mathbb{CP}^2$ are unbounded.
\end{abstract}

\section{Introduction}

A \textit{stabilization} refers to taking the connected sum of $X$ with $S^2\times S^2$. It is well known that stabilization can affect the smooth structure. For example, many families of exotic smooth 4-manifolds become mutually diffeomorphic after a single stabilization, as shown in \cite{akbulut2002variations} \cite{auckly2003families} \cite{baykur2013round}. Wall \cite{wall1964simply} proved that homeomorphic simply connected 4-manifolds become diffeomorphic after sufficiently many stabilizations. However, no examples are currently known of exotic 4-manifolds that remain exotic after one stabilization. Thus, it is natural to ask how stabilization interacts with smooth structures. Embedded surfaces provide one way to study this question, since they can retain information about the smooth structure.

A topologically embedded closed surface $F$ in a closed smooth oriented 4-manifold $X$ is called \textit{smoothable} (resp. \textit{nonsmoothable}) if there is a (resp. there is no) smooth surface topologically isotopic to $F$. In \cite[Theorem D]{cha2025light}, it is proved that any such $F$ is topologically isotopic to a smooth embedded surface after performing sufficiently many stabilizations. Cha and Kim pointed out that it is meaningful to study the \textit{smooth stabilization number}
\[
s_X(F):=\min\{k|\ F\ is\ smoothable\ in\ X\#k(S^2\times S^2)\}.
\]
The main theorem of the paper shows that in many cases, this number is 1:

\begin{them}\label{thm:main}
    Let $F$ be a locally flat nonsmoothable embedded closed simple surface in a smooth simply connected closed 4-manifold $X$. If $[F]\in H_2(X;\Z)$ is a primitive and ordinary class $($i.e. a class not characteristic$)$, then $F$ is smoothable in $X\#(S^2\times S^2)$. In particular, $s_X(F)=1$.
\end{them}

\noindent In the above theorem, the local flatness of $F$ means that for each $x\in F$, there is a neighborhood $U\subset X$ such that $(U, U\cap F)$ is homeomorphic to $(\R^4,\R^2)$. The surface $F$ is defined to be \textit{simple} if $\pi_1(X-F)$ is abelian. An element $a\in H_2(X;\Z)$ is called \textit{characteristic} if $\mathrm{PD}(a)\cdot x\equiv x^2\ (mod\ 2)$ for all $x\in H^2(X;\Z)$. A class $a\in H_2(X;\Z)$ is said to be \textit{primitive} if $a$ cannot be written as $a=dx$ unless $d=\pm1$ for $x\in H_2(X;\Z)$.

Theorem \ref{thm:main} contains some information of the stabilization in the smooth category of 4-manifolds. One related invariant is the \textit{smooth genus function} $g_X(\alpha)$, which is defined by the smallest genus of a smooth embedding representing the homology class $\alpha$ in the second homology of the closed smooth 4-manifold $X$. As in the examples in Section 3, there are manifolds with certain classes whose genus function is strictly positive but vanishes after a stabilization. This phenomena is a consequence of Wall \cite{wall1964diffeomorphisms}. In fact, \cite[Theorem 3]{wall1964diffeomorphisms} states that any primitive and ordinary class of $H_2(X\#S^2\times S^2)$ has embedded smooth spheres as a representative and these spheres are essentially from the standard spheres from the $S^2\times S^2$ component. In contrast, the examples in Section 3 show that it is the nonsmoothable 2-spheres becoming smoothable in the stabilization that makes the minimal genus vanishes. Thus, Theorem \ref{thm:main} provides a more intuitive explaination for the decreasing of the minimal genus function under stabilizations.

For another point of view, \textit{stabilizations does not eliminate all exotic phenomenon.} As a consequence of \cite[Example 1)]{auckly2019isotopy}, there is a nonsmoothable surface in a $K3$ surface such that after a single stabilization, the topological isotopic class of this surface contains infinitely many smooth representatives that are pairwisely not smoothly isotopic. See Proposition \ref{prop:1} for detail.\\

The proof of the main theorem consists of finding a smooth geometric dual (Definition \ref{def:geodual}) of $F_1$ in $X\#S^2\times S^2$. Once such a dual is found, $F_1$ is smoothable by \cite[Theorem C]{cha2025light}. Here $F_1$ is a surface topologically isotopic to $F$ and is nonsmooth at only one point. In fact, \cite[Theorem C]{cha2025light} claims that any topological surface with a smooth geometric dual is topologically isotopic to a smooth surface. The method of constructing a smooth geometric dual is essentially due to Auckly-Kim-Melvin-Ruberman-Schwartz \cite{auckly2019isotopy}.\\

In the next section, we give the proof of the main theorem. Examples of nonsmoothable surfaces becoming smoothable after a single stabilization are provided in Section 3. The difference of genus functions before and after the stabilization is given as well. In the last section, a nonsmoothable genus 8 surface $\Sigma$ in $\mathbb{CP}^2$ with $s_{\mathbb{CP}^2}(\Sigma)=2$ is constructed. In fact, there are surfaces in $\mathbb{CP}^2$ whose smooth stabilization number can be arbitrarily large. \textit{This implies that the map from the smooth isotopic class of smooth embedded surfaces to the topological isotopic class of topological embedded surfaces can still be non-surjective under stabilizations.}

By inspecting the proof of Theorem \ref{thm:main}, the requirements of $[F]$ being primitive and $F$ being a simple surface is to ensure that the meridian of $F_1$ at a point is contractible in $X-F_1$. This is necessary for constructing a smooth immersed sphere that intersecting $F_1$ only once so that the argument from \cite{auckly2019isotopy} can be applied. Details can be found in Section 2. Therefore, the following version of the main theorem is true as well.

\begin{theoremoneprime}
    Let $F$ be a locally flat nonsmoothable embedded closed surface in a smooth simply connected closed 4-manifold $X$. If $[F]\in H_2(X;\Z)$ is an ordinary class and the meridian of $F$ at a point is contractible in $X-F$, then $F$ is smoothable in $X\#(S^2\times S^2)$. In particular, $s_X(F)=1$.
\end{theoremoneprime}

\noindent\textbf{AI disclosure:} The author used ChatGPT only for exploring examples presented in Section 3 and grammar checking. The paper was written entirely by the author. The author takes full responsibility for the correctness of the paper.\\

\noindent\textbf{Acknowledgement:} The author would like to thank Professor Jae Choon Cha for introducing him to this topic during a short course at Peking University in 2026 and for patiently answering his questions. The author would like to thank Professor Jianfeng Lin for helpful suggestions.

\section{Proof of the main theorem}
The goal of this section is to prove Theorem \ref{thm:main}. In order to do this, we give the definition of a geometric dual:

\begin{defn}\label{def:geodual}
    A \textit{geometric dual} $G$ of a topological surface $F$ in a smooth 4-manifold without boundary is a 2-sphere such that
    \begin{itemize}
        \item $G$ intersects F transversely only once,
        \item the algebraic intersection number of $G$ with itself is $0$.
    \end{itemize}
\end{defn}

\begin{rmk}
    This definition is equivalent to the framed geometric dual in \cite{cha2025light}.
\end{rmk}

\noindent A geometric dual is useful in geometric topology in the sense that it helps to remove the intersections between surfaces.

\begin{proof}[Proof of Theorem \ref{thm:main}]
    The key of the proof is to find a smooth geometric dual. The starting point is to construct a smoothly immersed sphere that intersects the surface only once. To achieve this, we use the fact that $[F]$ is primitive and $F$ is simple.

    Since $[F]$ is primitive, using the relative long exact sequence and Thom isomorphism for simply connected $X$, one can show $H_1(X-F)=0$. In fact, if a class is of the form $dx$, where $x\in H_2(X;\Z)$ being primitive and $0\neq d\in\Z$, then $H_1(X-S;\Z)=\Z_d$ for any representative $S\in dx$. The calculation is carried out in \cite[Lemma 3.1]{hsiang1971embedding}. Because $F$ is primitive, $H_1(X-F;\Z)=\Z_d$ with $d=1$. Therefore, $H_1(X-F)=0$. Moreover, $F$ is simple, so $\pi_1(X-F)$ and $H_1(X-F;\Z)$ are isomorphic. As a result, $X-F$ is simply connected.

    In \cite[Theorem 1.7]{gompf2025topologically}, Gompf shows that any locally flat surface in a smooth 4-manifold is topologically isotopic to another surface that is smooth except at a singular point. We apply this theorem to $F$ to get a surface $F_1$ that is nonsmooth at only one point. Let $p\in F_1$ be a smooth point. Then the unit disc $D$ of the fiber of the normal bundle of $F_1$ at $p$ intersects $F_1$ only once. Because $\pi_1(X-F_1)=\pi_1(X-F)$ is trivial, $\partial D$ bounds a smoothly immersed disc in $X-F_1$. After gluing this immersed disc to $D$ along the common boundary, one gets a smoothly immersed 2-sphere $G_0$ such that $G_0\pitchfork F_1=p$.

    Under the assumption that $[F_1]=[F]$ is ordinary, the argument of the proof of the main theorem in \cite{auckly2019isotopy} implies that there is a smoothly embedded 2-sphere $G\subset X\#(S^2\times S^2)$, constructed out of $G_0$, that intersects $F_1$ exactly once and has self-intersection 0. The surface in \cite{auckly2019isotopy} is assumed to be smooth, but the argument still work for surfaces with only one nonsmooth point. {As a result, we construct a smooth geometric dual $G\subset X\#(S^2\times S^2)$ of a surface topologically isotopic to $F$}.


    With the smooth geometric dual $G$ of $F_1$, \cite[Theorem C]{cha2025light} can be applied to obtain a smooth embedded surface that is topologically isotopic to $F_1$ (thus topologically isotopic to $F$) in $X\#(S^2\times S^2)$. This finishes the proof that $F$ is smoothable in $X\#(S^2\times S^2)$.
\end{proof}

\noindent As mentioned in the proof of the main theorem in \cite{auckly2019isotopy}, the above argument holds for replacing $[F]$ ordinary and $X\#S^2\times S^2$ by $[F]$ characteristic and $X\#S^2\tilde\times S^2$, where $S^2\tilde\times S^2$ is the nontrivial $S^2$-bundle over $S^2$. So the following is true as well.

\begin{them}
    Let $F$ be a locally flat nonsmoothable embedded closed simple surface in a smooth simply connected closed 4-manifold $X$. If $[F]\in H_2(X;\Z)$ is a primitive and characteristic class, then $F$ is smoothable in $X\#(S^2\tilde\times S^2)$.
\end{them}

\section{Examples and the smooth genus function}

As a consequence of Theorem \ref{thm:main}, it turns out that there are many nonsmoothable surfaces that become smoothable after a single stabilization. Therefore, the smooth stabilization number of these surfaces are all 1.

\begin{ex}[$K3$ surfaces]
    In \cite[Corollary 1.4]{lee1997representing}, Lee and Wilczy\'nski showed that for a nonnegative self-intersection class $x\in H_2(K3)$, there is a nonsmoothable embedded genus $g$ locally flat simple surface $\Sigma$ in this class if $g<\frac12x\cdot x+1$. Recall that the intersection form of the $K3$ surface is $-E_8\oplus-E_8\oplus3\begin{bmatrix}0&1\\1&0\end{bmatrix}$ and let $(e_i,f_i)$ be the basis for the $i$-th hyperbolic summand $\begin{bmatrix}0&1\\1&0\end{bmatrix}$. Specifically,
    \[
    e_i^2=f_i^2=0,\ e_i\cdot f_i=1,\ i=1,2,3.
    \]
    Therefore, there are infinitely many classes in $H_2(K3;\Z)$ that are primitive and ordinary. For example, one can take $x_n:=n\mathrm{PD}(e_1+f_1)+\mathrm{PD}(e_2)$ for $n\in\Z_{>0}$. These are all primitive and ordinary since there is an $e_2$ summand in $x_n$ and $\mathrm{PD}(f_2)^2=0$ but $\mathrm{PD}(f_2)\cdot x_n=1$.  Clearly $x_n^2=2n^2$, thus, there is a nonsmoothable embedded locally flat simple surface $\Sigma_n$ in each class $x_n$ according to \cite[Corollary 1.4]{lee1997representing} as long as the genus of $\Sigma_n$ is smaller than $n^2+1$. Then Theorem \ref{thm:main} implies that $\Sigma_n$ is smoothable in $K3\#(S^2\times S^2)$. As a result, \textit{we construct infinitely many nonsmoothable embedded surfaces that are smoothable after a stabilization in a $K3$ surface, and the embedded surface can be chosen with arbitrary genus.}

    As calculated in \cite{hamilton2014minimal}, the genus function
    \[
    g_{K3}(x_n)=n^2+1.
    \]
    But it is clear that this number becomes 0 after one stabilization.
\end{ex}

The following proposition shows that after a single stabilization, the nonsmoothable surface may become smoothable but its topological isotopy class may contain infinitely many smooth surfaces that are pairwisely not smoothly isotopic.

\begin{prop}\label{prop:1}
    Let $\Sigma_n$ be the nonsmoothable embedded 2-sphere in the class $x_n\in H_2(E(2))$ defined in the last example. Then the topological isotopy class of $\Sigma_n$ in $E(2)\#S^2\times S^2$ contains infinitely many embedded smooth 2-spheres that are not smoothly isotopic to each other.
\end{prop}

\begin{proof}
    Since $x_n^2$ is even and non-negative, then \cite[Example 1)]{auckly2019isotopy} implies directly that there are infinitely many embedded smooth 2-spheres, denoted as $\{\Sigma_n^j\}_{j=1}^\infty$, in $E(2)\#S^2\times S^2$ that are not smoothly isotopic to each other in the class $x_n$ and $\pi_1(E(2)\#S^2\times S^2-\Sigma_n^j)=1$. As claimed in \cite{galvin2026smooth}, homologous surfaces with simply connected complements in a simply connected 4-manifold are topologically isotopic according to \cite[Theorem F]{boyer1993realization}. The simply connectedness of $E(2)\#S^2\times S^2-\Sigma_n^j$ is a direct consequence of the construction in the proof of the only corollary in \cite{auckly2019isotopy} for $j=1,2\ldots$. Therefore, $\Sigma_n$ and $\Sigma_n^j$ are topologically isotopic.
\end{proof}

The next two examples need the following theorem from \cite[Theorem 1.1]{lee1990locally}.

\begin{them}[{\cite[Theorem 1.2]{lee1990locally}}]\label{thm:lw1990}
Let $X$ be a smooth closed 4-manifold and let $x\in H_2(X;\Z)$ be a primitive and ordinary class. Then there exists a locally flat simple embedded 2-sphere in $X$ representing $x$ if and only if 
\[
b_2(X)\ge\sigma(X),
\]
where $\sigma(X)$ is the signature of $X$.
\end{them}

\begin{rmk}
    What proved in \cite[Theorem 1.2]{lee1997representing} is stronger and what is stated here is a simplified version.
\end{rmk}

\begin{ex}[$S^2\times S^2$]
    In \cite[Theorem 1]{kuga1984representing}, it is proved that $pe_1+qe_2\in H_2(S^2\times S^2;\Z)$ can be represented by a smoothly embedded $2$-sphere if and only if $|p|\le1$ or $|q|\le1$, where $p,q\in\Z$ and $e_1,e_2$ are the standard generators of $H_2(S^2\times S^2;\Z)$. Because $pe_1+qe_2\in H_2(S^2\times S^2;\Z)$ is primitive and ordinary for any coprime pair $(p,q)$ with $|p|,|q|>1$, Theorem \ref{thm:lw1990} implies that there is a locally flat simple embedded 2-sphere representing $pe_1+qe_2$. Therefore, this 2-sphere cannot be smooth according to \cite[Theorem 1]{kuga1984representing}. But this 2-sphere is smoothable in $S^2\times S^2\#S^2\times S^2$ by Theorem \ref{thm:main}. As a result, \textit{we obtain infinitely many nonsmoothable embedded 2-spheres in $S^2\times S^2$ that are smoothable after a stabilization.}

    As calculated in \cite{ruberman1996minimal}, the genus function
    \[
    g_{S^2\times S^2}(pe_1+qe_2)=(|p|-1)(|q|-1)>0\ for\ |p|,|q|>1.
    \]
    But it is clear that this number becomes 0 after one stabilization.
\end{ex}

\noindent We conclude this section with examples in $\mathrm{CP}^2\#\overline{\mathrm{CP}^2}$.

\begin{ex}[$\mathrm{CP}^2\#\overline{\mathrm{CP}^2}$]
    In the main theorem of \cite{luo1988representing}, it is proved that $ae_1+be_2\in H_2(\mathrm{CP}^2\#\overline{\mathrm{CP}^2};\Z)$ can be represented by a smoothly embedded $2$-sphere if and only if either $||a|-|b||\le1$ or $(a,b)=(\pm2,0)\ or\ (0,\pm2)$, where $a,b\in\Z$ and $e_1,e_2$ are the standard generators of $H_2(\mathrm{CP}^2\#\overline{\mathrm{CP}^2};\Z)$. Because $2e_1+be_2\in H_2(\mathrm{CP}^2\#\overline{\mathrm{CP}^2};\Z)$ is primitive and ordinary for any odd number $|b|>3$, Theorem \ref{thm:lw1990} implies that there is a locally flat simple embedded 2-sphere representing $2e_1+be_2$. Therefore, this 2-sphere cannot be smooth according to the main theorem of \cite{luo1988representing}. But this 2-sphere is smoothable in $\mathrm{CP}^2\#\overline{\mathrm{CP}^2}$ by Theorem \ref{thm:main}. As a result, \textit{we obtain infinitely many nonsmoothable embedded 2-spheres in $\mathrm{CP}^2\#\overline{\mathrm{CP}^2}$ that are smoothable after a stabilization.}

    As calculated in \cite{ruberman1996minimal}, the genus function
    \[
    g_{\mathrm{CP}^2\#\overline{\mathrm{CP}^2}}(2e_1+be_2)=\frac{|b|(|b|-3)}{2}.
    \]
    But it is clear that this number becomes 0 after one stabilization.
\end{ex}

\section{Surfaces whose nonsmoothability is preserved under stabilizations}
This section we show that there are nonsmoothable embedded surfaces that remain nonsmoothable under stabilizations.

\begin{them}[{\cite[Theorem 1.1]{lee1997representing}}]\label{thm:LW97}
    Let $X$ be a closed oriented simply connected smooth 4-manifold. Suppose $x\in H_2(X)$ is a nonzero class of divisibility $d$. Then there exists a simple topologically locally flat embedding genus $g>0$ surface if and only if
    \[
    b_2(X)+2g\ge \max_{0\le j<d}|\sigma(X)-2j(d-j)\frac1{d^2}x\cdot x|.
    \]
\end{them}

The following theorem is also needed for the existence of surfaces whose nonsmoothability is preserved under stabilizations.

\begin{them}[{\cite[Theorem 1.6]{bryan1998seiberg}}]
    Let $X$ be a closed oriented simply connected smooth 4-manifold with $b_2^+(X)>1$, and let $\Sigma$ be a smoothly embedded genus $g$ surface. Suppose that the homology class $[\Sigma]$ is divisible by 2 and that $[\Sigma]/2\equiv w_2(X)$ $\mathrm{mod}$ $2$. Then
    \[
    g(\Sigma)\ge\frac54(\frac{[\Sigma]^2}{4}-\sigma(X))-b_2(X)+2.
    \]
\end{them}

The following theorem shows that for $\mathbb{CP}^2$, there are topological embedded surfaces with its $s_{\mathbb{CP}^2}$ arbitrary large.

\begin{them}
    Let $h\in H_2(\mathbb{CP}^2)$ be the generator. Then for every class $2nh$, there exists a nonsmoothable embedded surface $\Sigma_n$ representation with genus $g(\Sigma_n)=n^2-1$ such that 
    \[
    s_{\mathbb{CP}^2}(\Sigma_n)\ge \frac{n^2+3}{8},
    \]
    where $n\ge3$ is an integer.
\end{them}

\begin{proof}
    According to Theorem \ref{thm:LW97}, the existence of $\Sigma_n$ is equivalent to
    \[
    1+2g(\Sigma_n)\ge\max_{0\le j<2n}|1-2j(2n-j)\frac{1}{4n^2}4n^2|=2n^2-1, \ n\ge3.
    \]
    Since $g(\Sigma_n)=n^2-1$, such an embedded surface exists. The smallest genus of the class $2nh$ is $\frac{(2n-1)(2n-2)}{2}$. It requires that 
    \[
    n^2-1\ge2n^2-3n+1,\ for\ n\ge3,
    \]
    which is impossible. Moreover, if $\Sigma_n$ is smooth after $k$-times of stabilizations, the previous theorem implies that
    \[
    n^2-1\ge\frac54(n^2-1)-2k-1+2=\frac{5n^2}{4}-2k-\frac14.
    \]
    Therefore,
    \[
    k\ge \frac{n^2+3}{8}.
    \]
\end{proof}

The next example is a surface $\Sigma$ in $\mathbb{CP}^2$ with $s_{\mathbb{CP}^2}(\Sigma)=2$.

\begin{ex}
    Let $n=3$ in the previous theorem, then there is a nonsmoothable surface $\Sigma$ with $g(\Sigma)=8$ and $s_{\mathbb{CP}^2}(\Sigma)\ge\frac{12}{8}$. Therefore,
    \[
    s_{\mathbb{CP}^2}(\Sigma)\ge2.
    \]
    Note that $[\Sigma]=6h$, where $h\in H_2(\mathbb{CP}^2)$ is the generator, so there is no geometric dual for $\Sigma$. There is a smoothly embedded surface $\Sigma_{10}$ with genus $10$ in the class $6h$ \cite[Section 2.1]{gompf19994}. We show that each time performing a stabilization, there is a compressing disc with its boundary in $\Sigma_{10}$ to reduce the genus by 1. By computing the fundamental group of the surface exterior after twice stabilizations, it can be shown that the surface is topologically isotopic to $\Sigma$.

    Using the same calculation as in \cite[Lemma 3.1]{hsiang1971embedding},
    \[
    \pi_1(\mathbb{CP}^2-\Sigma_{10})=\Z_6.
    \]
    With this, we need the key lemma below:

    \begin{lem}\label{lem:compre}
        Let $\Sigma'$ be a closed oriented smoothly embedded surface in $\mathbb{CP}^2\#kS^2\times S^2$ with genus $g(\Sigma')\ge2$ such that 
        \[
        [\Sigma']=6h,\ \pi_1(\mathbb{CP}^2\#kS^2\times S^2-\Sigma')=\Z_6,
        \]
        where $k=0,1$. Then there is a closed oriented smoothly embedded surface $\Sigma''$ in $\mathbb{CP}^2\#(k+1)S^2\times S^2$ such that 
        \[
        [\Sigma'']=6h,\ g(\Sigma'')=g(\Sigma')-1,\ \pi_1(\mathbb{CP}^2\#(k+1)S^2\times S^2-\Sigma')=\Z_6.
        \]
    \end{lem}

    After applying this lemma twice, we have a genus 8 surface $\Sigma_{8}$ in the class $6h\in H_2(\mathbb{CP}^2\#2S^2\times S^2)$. By Theorem \ref{thm:LW97}, the minimal genus of a locally flat embedded surface representing the class $6h$ is $6<8$. As a result, \cite[Theorem 7.4]{sunukjian2015surfaces} implies that $\Sigma$ and $\Sigma_8$ are topologically isotopic. This proves that 
    \[
    s_{\mathbb{CP}^2}(\Sigma)=2.
    \]

    It remains to justify the lemma.

    \begin{proof}[Proof of Lemma \ref{lem:compre}]
        The main idea is to find a compressing disc to reduce the genus of $\Sigma'$ after the stabilization. Let $P$ be a genus 2 subsurface of $\Sigma'$ with only one boundary. The first step is to find a vector field in the normal direction such that the push-off the loops in $P$ along that direction are all contractible in $X_k-\Sigma'$, where $X_k:=\mathbb{CP}^2\#kS^2\times S^2$. Tracking down the calculation of $\pi_1(X_k-\Sigma')$ \cite[Lemma 3.1]{hsiang1971embedding}, the generator of this group is the meridian of the surface, denoted as $\mu$. Then, for the generator $a_i$ of $\pi_1(P)$ with $i=1,2,3,4$, its push-off along a given normal direction is represented by $\mu^{m_i}$ in $\pi_1(X_k-\Sigma')$. After choosing another normal direction $v$ that winding $-m_i$ times around the meridian of $P$ with respect to the previous given one along $a_i$, the push-off of $a_i$ is contractible in $X_k-\Sigma'$.

        The second step is to determine the push-off loop to perform a surgery such that the resulting manifold is $X_{k+1}:=\mathbb{CP}^2\#(k+1)S^2\times S^2$. As a matter of fact, doing surgery with respect to the framing induced by the bounding disc of the contractible loop produces a stabilization of the original manifold. What matters here is whether the framing induced by the push-off and this disc framing coincide or not. This is necessary for the next step. Since $P$ has a boundary, the normal direction $v$ can be extended over $P$. Without loss of generality, let $a_1$ and $a_2$ be two disjoint linearly independent loops in $\pi_1(P)$. Then the normal framing of $a_i$ in $X_k$ induced by $P$ can be moved to a normal framing of $b_i$ with $i=1,2$. These framings are called the surface framings. Comparing the surface framing and the disc framing along the push-off loops $b_1$ and $b_2$, one gets elements $r_1$ and $r_2$ in $\pi_1(SO(3))\cong\Z_2$, respectively. If either $r_1$ or $r_2$ are trivial in $\Z_2$, then the loop corresponding to the trivial element is chosen to be the surgery loop. If both $r_1$ and $r_2$ are non-zero, there is a band induced by pushing off a thickening a path connecting $a_1$ ad $a_2$ in $P$ along $v$. The surgery loop in this case is the connected sum of $b_1$ and $b_2$ along this band. The difference between the surface framing and the disc framing along this loop is $r_1+r_2=0$ in $\Z_2$.

        The third step is to find a compressing disc in the stabilized manifold to reduce the genus of $\Sigma'$. Note that $X_{k+1}$ is constructed by removing neighborhood of the surgery loop and then glue $D^2\times S^2$ back. Then the union of $D^2\times\{pt\}$ and the band used to push off the loop (Figure \ref{fig:surgeryband}) is a compressing disc $D$ because the framing matches as chosen in the previous step. Therefore, glue $D\times[-1,1]$ to $\Sigma'$ is the oriented smoothly embedded surface $\Sigma''$ we are seeking for. By the construction, $[\Sigma'']\in 6h$.

        It remains to show that
        \[
        \pi_1(X_{k+1}-\Sigma')=\Z_6.
        \]
         The method is to apply the van Kampen theorem. Let $B$ be a 4-dimensional ball thickening the glued handle $D\times[-1,1]$ a little. Then $B\cap\Sigma'$ is an annulus and $B\cap\Sigma''$ consists of two parallel discs. Meanwhile,
         \[
         \partial(B-\Sigma')=\partial(B-\Sigma''),
         \]
         with the fundamental group being the free group $F(e_1,e_2)$ generated by $e_1$ and $e_2$. Notice that 
         \[
         \pi_1(B-\Sigma')=\Z,\ \pi_1(B-\Sigma'')={F(e_1,e_2)},\ X_{k+1}-\Sigma'-B=X_{k+1}-\Sigma''-B:=E,
         \]
         where the generator of $\pi_1(B-\Sigma')$ is the meridian of the annulus and the generators of $\pi_1(B-\Sigma'')$ are the meridians of the two discs. According to the van Kampen theorem,
         \begin{align*}
             \pi_1(X_{k+1}-\Sigma')&=\pi_1(E)*_{F(e_1,e_2)}\Z=\Z_6,\\
             \pi_1(X_{k+1}-\Sigma'')&=\pi_1(E)*_{F(e_1,e_2)}F(e_1,e_2)\cong\pi_1(E),
         \end{align*}
         where the first amalgamation is by identifying $e_1$ and $e_2$ with the generator of $\Z$ and the second amalgamation is the identity map. By comparing these two amalgamations,
         \[
         \pi_1(E)/e_1e_2^{-1}\cong\Z_6.
         \]
         However, $e_1$ and $e_2$ both are the meridians of $\Sigma''$. Therefore, $e_1e_2^{-1}=1$ automatically. This shows that $\pi_1(X_{k+1}-\Sigma'')\cong\Z_6$.
    \end{proof}

    \begin{figure}[H]
    \centering 
    \includegraphics[width=0.8\textwidth]{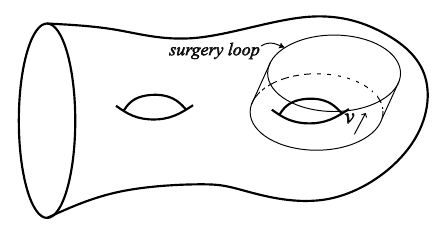}
    \caption{The band that used to push off the surgery loop.}\label{fig:surgeryband}
     \end{figure}
\end{ex}

\bibliographystyle{alpha}
\bibliography{references}

\noindent Zuyi Zhang, 
Beijing International Center for Mathematical Research, 
 Peking University, Beijing, 100871, China\\
\noindent Email: zhangzuyi1993@pku.edu.cn\\

\end{document}